\documentclass[11pt,letterpaper]{amsart}

\numberwithin{equation}{section}
\theoremstyle{plain}
\newtheorem{theorem}{Theorem}[section]
\newtheorem{proposition}[theorem]{Proposition}

\newtheorem{lemma}[theorem]{Lemma}
\newtheorem*{conjecture}{Conjecture}

\theoremstyle{definition}

\newtheorem{remark}[theorem]{Remark}

\newcommand{\Rmnum}[1]{\expandafter\@slowromancap\romannumeral #1@}

\newcommand{\mr}{\mathbb{R}}
\newcommand{\ud}{\mathrm{d}}
\allowdisplaybreaks


\numberwithin{equation}{section}

\usepackage{hyperref}

\address{K. Li, School of Mathematics and Statistics, Zhengzhou University, Zhengzhou, China}
\email{likui@zzu.edu.cn}
\address{M. Li, Department of Mathematics, Chinese University of Hong Kong, Shatin, NT, Hong Kong}
\email{mingxiangli@cuhk.edu.hk}
\address{J.  Wei, Department of Mathematics, Chinese University of Hong Kong, Shatin, NT, Hong Kong}
\email{wei@math.cuhk.edu.hk}
\begin{document}
	
	\title[Lane-Emden conjecture]{On a new region for  the  Lane-Emden conjecture in higher dimensions}
	\author{Kui Li, Mingxiang Li, Juncheng Wei}
	\date{}
	\maketitle
	
	\begin{abstract}
		We study the Lane-Emden conjecture, which asserts the non-existence of non-trivial, non-negative solutions to the Lane-Emden system
		\begin{equation*}
			\left\{
			\begin{aligned}
				-\Delta u &= v^p, \quad x \in \mathbb{R}^n, \\
				-\Delta v &= u^q, \quad x \in \mathbb{R}^n,
			\end{aligned}
			\right.
		\end{equation*}
		in the subcritical regime. By employing an Obata-type integral inequality, Picone's identity, and exploiting the scaling invariance of the system, we prove that   the conjecture holds for any dimension $n \geq 5$ and  exponents satisfying $p\geq 1,q\geq 1$, and
		\[
		\frac{1}{p+1} + \frac{1}{q+1} \geq 1 - \frac{2}{n} + \frac{4}{n^2}.
		\]
	\end{abstract}

	\section{Introduction}

	The following nonlinear elliptic equation 
	\begin{equation}\label{leequ}
		-\Delta u = u^p,\quad x \in \mathbb{R}^n
	\end{equation}
	has been extensively studied and serves a fundamental role in the theory of both elliptic and parabolic partial differential equations.   A classical and seminal result is due to Gidas and Spruck \cite{GS}, who established that the equation \eqref{leequ} does not admit any  positive classical solution in the subcritical range $0<p<p_S$	, where $p_S$ is the Sobolev exponent given by
	\begin{equation*}
		p_S :=
		\begin{cases}
			\infty, & \text{if } n = 2, \\
			\frac{n+2}{n-2}, & \text{if } n \geq 3.
		\end{cases}
	\end{equation*}
	Subsequently, Chen and Li \cite{CL1} (see also Li and Zhu \cite{LZ1995})provided an  alternative proof using the Kelvin transform and the method of moving planes, which further illuminated the geometric and symmetric properties of the problem.

	Such nonexistence theorems not only reveal the structure of positive solutions but also have wide applications in blow-up analysis, a priori estimates, and classification problems. We refer the reader to \cite{BC, BG, BV, BM, CGS, CL1, DP, F, GNN, GS, Phan-Soup, QS} and the references therein for further details.

	A natural and important generalization of \eqref{leequ} is the following Lane-Emden system:
	\begin{equation}\label{lesys}
		\begin{cases}
			-\Delta u = v^p, & x \in \mathbb{R}^n, \\
			-\Delta v = u^q, & x \in \mathbb{R}^n.
		\end{cases}
	\end{equation}
	This system has been extensively studied in recent decades, and the existence or nonexistence of positive solutions depends delicately on the exponents $p,q$ and the dimension $n$. Liouville-type theorems for such systems continue to be a central topic in nonlinear analysis, with deep connections to problems in geometry, mathematical physics and mathematical biology.

	For given positive constants $p$ and $q$, the pair $(p,q)$ is referred to as subcritical if it lies below the Sobolev hyperbola, that is,
	\begin{equation}\label{subcri}
		\frac{1}{p+1} + \frac{1}{q+1} > 1 - \frac{2}{n}.
	\end{equation}
	When $pq\not=1$, we define
	\begin{equation}\label{alphabeta}
		\alpha := \frac{2(p+1)}{pq - 1}, \quad \beta := \frac{2(q+1)}{pq - 1}.
	\end{equation}
	In the case $pq>1$,  the condition \eqref{subcri} is equivalent to
	\begin{equation*}
		\alpha + \beta > n - 2.
	\end{equation*}
	Assume $pq\not=1$ and let $(u,v)$ be a solution of system \eqref{lesys}. Then for any $R>0$, the rescaled pair
	$$(R^{\alpha} u(Rx), R^{\beta} v(Rx))$$
	also satisfies system \eqref{lesys}. This scaling invariance will  play a fundamental role during the proof  in this paper.

	One of the most interesting and important questions concerning the classification of solutions of system \eqref{lesys} is the following  \emph{Lane-Emden conjecture}.

	\begin{conjecture}
		\emph{Suppose that  $p,q>0$. If the pair $(p,q)$ is subcritical \eqref{subcri}, then the system \eqref{lesys}
			has no positive classical solutions.}
	\end{conjecture}

	We now briefly review some key developments related to this conjecture.
	
	Mitidieri \cite{Miti,M} established that system \eqref{lesys} admits no positive radial solutions if and only if the pair $(p,q)$ is subcritical, thereby confirming the conjecture in the case of positive radial solutions (see also \cite{PQS}, \cite{SZ1}).
	
	Serrin and Zou \cite{SZ1} proved that the conjecture holds true under polynomial growth assumptions when either $pq \leq 1$, or $pq > 1$ and $\max\{\alpha, \beta\} \geq n-2$. De Figueiredo and Felmer \cite{FF} verified the Lane-Emden conjecture for $0 < p, q \leq \frac{n+2}{n-2}$ with $(p,q) \neq (\frac{n+2}{n-2}, \frac{n+2}{n-2})$, using the moving plane method. Busca and Manásevich \cite{BM} obtained another partial result under the condition that $\alpha, \beta \in \left[\frac{n-2}{2}, n-2\right)$ and $(\alpha,\beta) \neq \left(\frac{n-2}{2}, \frac{n-2}{2}\right)$. 
	
	For the special case $p=1$, it has been solved by Lin \cite{L} and Wei-Xu \cite{WX} by using moving plane method. Very recently, Ma, Wu and Wu \cite{MWW} extended such  result to manifolds with non-negative Ricci curvature by using invariant tensor technique.
	
	Under suitable integral growth assumptions, the conjecture was  resolved in the works of Chen-Li \cite{CL2},  Cheng-Huang \cite{CHL} and Ma-Chen \cite{MC}.
	
	In a different direction, Poláčik, Quittner, and Souplet \cite{PQS} introduced a Doubling Lemma and showed that for $pq > 1$, the non-existence of bounded positive classical solutions of system \eqref{lesys} implies the non-existence of all positive classical solutions. Using this approach, they confirmed the conjecture in dimension $n = 3$. Souplet \cite{S} later extended this result to the case where $\max\{\alpha, \beta\} > n - 3$, which in particular settles the conjecture for dimension $n = 4$. For further discussions, see \cite{BM1, CFM, DLY, FS, LZ, LS, Miti, RZ, SZ2, S1, Z} and the references therein.
	
	Consequently, for the purpose of establishing the Lane-Emden conjecture, it suffices to consider bounded solutions under the condition $pq>1$.
	By employing an Obata-type integral inequality, Picone's identity, and the scaling invariance of the system, we prove the conjecture for a new range of exponents, as stated in the following theorem.
	\begin{theorem}\label{mainthm}
		Suppose  that $p\geq1$ and   $q\geq 1$.  For any  $n\geq 5$ and
		\begin{equation}\label{new regionin main theorem}
			\quad \frac{1}{p+1}+\frac{1}{q+1}\geq 1-\frac{2}{n}+\frac{4}{n^2},
		\end{equation}
		then system \eqref{lesys} has no positive solutions.
	\end{theorem}
	\begin{remark}\label{generalthm}
		In fact, we can improve the constant $4$ in condition \eqref{new regionin main theorem} to be any positive constant $c_0>2$. Then Lane-Emden conjecture holds true for large enough dimension $n$  which depends on $c_0$ (see Theorem \ref{mainthm for c_0}).
	\end{remark}
	
	Now, let us briefly outline our strategy and the structure of this paper.  The entire argument in this paper proceeds by contradiction.
	
	First, consider the single equation \eqref{leequ}. In \cite{GS}, the Obata-type integral estimate (See Lemma \ref{obata-inequality}) plays a fundamental role in the proof. By choosing appropriate coefficients, Liouville-type results are derived. For the system \eqref{lesys}, analogous estimates can be established for $u$ and $v$ separately. The first major difficulty lies in handling the coupled term
	$\int_{\mr^n} u^{r-1}|\nabla u|^2v^p\ud x.$ To address this, we apply a simple yet powerful tool--Picone’s identity (Lemma \ref{picone-identity})—along with integration by parts, leading to the key estimates in Theorem \ref{estimate-theorem}.  We establish these estimates in Section \ref{sec:pre}.

	Once such system estimates are obtained, the second challenge is to select suitable coefficients in Theorem \ref{estimate-theorem} so as to derive universal growth estimates for $u$ and $v$ over the unit ball. In Section \ref{sec:choice}, we parameterize
	$p$ and $q$ as
	$$p=\frac{n+2d_1}{n-2d_1},\quad q=\frac{n+2d_2}{n-2d_2}$$
	An important advantage of this parametrization is that the conditions
	$p\geq 1,q\geq 1$, and \eqref{subcri} are transformed into $d_1\geq 0, d_2\geq 0, d_1+d_2<2$.  Using these parameters, we select appropriate coefficients to establish a crucial inequality \eqref{beta1beta2>gamma1gamma2} under the condition
	$d_1\geq 0, d_2\geq 0, d_1+d_2\leq 2-\frac{4}{n}$. As a corollary, we establish Theorem \ref{mainthm for c_0} by  leveraging the scale invariance property of the system.

	Finally, in Section \ref{sec:numercial}, we complete the proof of Theorem \ref{mainthm} through a suitable choice of coefficients and numerical computations carried out by Mathematica.

	\section{Preliminaries and key estimates}\label{sec:pre}
	
	In this section, we present several fundamental results that will be used in subsequent discussions firstly. These results are well established in the literature; for brevity, we state them below and omit the detailed proofs.
	\begin{lemma}\label{component-comparison}(See \cite{Miti},\cite{PS},\cite{SZ1}, Lemma 2.7 in \cite{S})
		Let $p\geq q>0$ with $pq>1$, $(u,v)$ be a positive solution of system \eqref{lesys} and assume that either $v$ is bounded or $p\geq2$. Then there holds
		$$\frac{v^{p+1}}{p+1}\leq \frac{u^{q+1}}{q+1},\quad x\in\mathbb{R}^n.$$
	\end{lemma}
	The following Obata-type integral inequality will be used later. It can be obtained by using Bocher formula and integration by parts. For brevity, the differential $\ud x$ will be omitted from all integrals throughout the rest of this paper.
	\begin{lemma} ( See Lemma 8.9 in \cite{QS})\label{obata-inequality}
		Suppose that $\Omega\subseteq\mathbb{R}^n$ and $u\in C^2(\Omega)$ is a positive function. Then for any $r,k\in \mathbb{R}$ with $k\neq-1$ and for any $\varphi\in C_c^2(\Omega)$ with $\varphi\geq 0$, we have
		\begin{equation*}
			\lambda_1\int_\Omega u^{r-2}|\nabla u|^4\varphi+\lambda_2\int_\Omega u^{r-1}|\nabla u|^2(-\Delta u)\varphi-\frac{n-1}{n}\int_\Omega u^r(-\Delta u)^2\varphi\leq I_1,
		\end{equation*}
		where $\lambda_1=-\frac{n-1}{n}k^2+k(r-1)-\frac{1}{2}r(r-1)$, $\lambda_2=\frac{3}{2}r-\frac{n+2}{n}k$ and
		\begin{equation*}
			I_1=\frac{1}{2}\int_\Omega u^r|\nabla u|^2\Delta\varphi+\int_\Omega u^r\Delta u\nabla u\cdot\nabla\varphi+(r-k)\int_\Omega u^{r-1}|\nabla u|^2\nabla u\cdot\nabla\varphi.
		\end{equation*}
	\end{lemma}
	
	A simple yet powerful technique is central to our approach for handling the coupled terms. We state it as a lemma due to its importance, although its proof is straightforward.
	\begin{lemma} (Picone's identity)\label{picone-identity}
		Suppose that $\Omega\subseteq\mathbb{R}^n$ and $u, v\in C^2(\Omega)$ with $v$ being positive. Then we have
		\begin{equation*}
			|\nabla u|^2-\nabla \big(\frac{u^2}{v}\big)\cdot\nabla v=|\nabla u-\frac{u}{v}\nabla v|^2.
		\end{equation*}
	\end{lemma}

	With the help of the above preparations, we present key estimates that will be instrumental in our subsequent analysis.
	\begin{lemma}\label{estimates-tansform}
		Suppose that $r,s\in\mathbb{R}$ with $r\not \in \{-1,-2\}$. If $(u,v)$ is a pair of positive solutions to system \eqref{lesys}, then for any $\varphi\in C_c^2(\mathbb{R}^n)$ we have
		\begin{equation}\label{estimates-tansform-1}
			\begin{aligned}
				&\int_{\mathbb{R}^n} u^rv^s|\nabla u|^2\varphi-\frac{s(s-1)}{(r+1)(r+2)}\int_{\mathbb{R}^n} u^{r+2}v^{s-2}|\nabla v|^2\varphi\\
				=&\frac{1}{r+1}\int_{\mathbb{R}^n} u^{r+1}v^{p+s}\varphi-\frac{s}{(r+1)(r+2)}\int_{\mathbb{R}^n} u^{r+q+2}v^{s-1}\varphi+I_2,
			\end{aligned}
		\end{equation}
		where
		\begin{equation*}
			I_2=-\frac{1}{r+1}\int_{\mathbb{R}^n} u^{r+1}v^s\nabla u\cdot\nabla\varphi+\frac{s}{(r+1)(r+2)}\int_{\mathbb{R}^n} u^{r+2}v^{s-1}\nabla v\cdot \nabla\varphi.
		\end{equation*}
	\end{lemma}
	
	\begin{proof} By using the divergence theorem and integration by parts, we have
		\begin{equation}\label{estimates-tansform-3}
			\begin{aligned}
				\int_{\mathbb{R}^n} u^rv^s|\nabla u|^2\varphi=&\frac{1}{r+1}\int_{\mathbb{R}^n}v^s\nabla\big(u^{r+1}\big)\cdot\nabla u\varphi\\
				=&-\frac{1}{r+1}\int_{\mathbb{R}^n} u^{r+1}\big(v^s\Delta u+sv^{s-1}\nabla u\cdot\nabla v\big)\varphi\\
				&-\frac{1}{r+1}\int_{\mathbb{R}^n} u^{r+1}v^s\nabla u\cdot\nabla\varphi.
			\end{aligned}
		\end{equation}
		Since $(u,v)$ satisfies the system \eqref{lesys}, the above identity \eqref{estimates-tansform-3} can be written as
		\begin{equation}\label{estimates-tansform-4}
			\begin{aligned}
				\int_{\mathbb{R}^n} u^rv^s|\nabla u|^2\varphi=&\frac{1}{r+1}\int_{\mathbb{R}^n} u^{r+1}v^{p+s}\varphi-\frac{s}{r+1}\int_{\mathbb{R}^n} u^{r+1}v^{s-1}\nabla u\cdot \nabla v\varphi\\
				&-\frac{1}{r+1}\int_{\mathbb{R}^n} u^{r+1}v^s\nabla u\cdot\nabla\varphi.
			\end{aligned}
		\end{equation}
		
		Now, we deal with the second term of the right hand side of \eqref{estimates-tansform-4}. Due to the same reason as before,  a direct computation yields that
		\begin{equation}\label{estimates-tansform-5}
			\begin{aligned}
				&\int_{\mathbb{R}^n} u^{r+1}v^{s-1}\nabla u\cdot\nabla v\varphi=\frac{1}{r+2}\int_{\mathbb{R}^n} v^{s-1}\nabla\big(u^{r+2}\big)\cdot \nabla v\varphi\\
				=&-\frac{1}{r+2}\int_{\mathbb{R}^n} u^{r+2}\big(v^{s-1}\Delta v+(s-1)v^{s-2}|\nabla v|^2\big)\varphi-\frac{1}{r+2}\int_{\mathbb{R}^n} u^{r+2}v^{s-1}\nabla v\cdot\nabla\varphi\\
				=&\frac{1}{r+2}\int_{\mathbb{R}^n} u^{r+q+2}v^{s-1}\varphi-\frac{s-1}{r+2}\int_{\mathbb{R}^n} u^{r+2}v^{s-2}|\nabla v|^2\varphi-\frac{1}{r+2}\int_{\mathbb{R}^n} u^{r+2}v^{s-1}\nabla v\cdot\nabla\varphi.
			\end{aligned}
		\end{equation}
		Combining \eqref{estimates-tansform-5}  with \eqref{estimates-tansform-4}, we obtain the identity \eqref{estimates-tansform-1} and finish the proof of this lemma.
	\end{proof}

	\begin{lemma}\label{estimates-picone}
		Suppose that $r,s, \theta\in\mathbb{R}$ with $r\neq-2$. If $(u,v)$ is a pair of positive solutions  to the  system \eqref{lesys},  for  any non-negative  $\varphi\in C_c^2(\mathbb{R}^n)$, there holds
		\begin{equation*}
			\begin{aligned}
				\int_{\mathbb{R}^n} u^rv^s|\nabla u|^2\varphi&\geq \frac{2\theta}{r+2}\int_{\mathbb{R}^n} u^{r+q+2}v^{s-1}\varphi\\
				&-\frac{2\theta}{r+2}\big(\theta+s+\frac{\theta r}{2}-1\big)\int_{\mathbb{R}^n} u^{r+2}v^{s-2}|\nabla v|^2\varphi+I_3,
			\end{aligned}
		\end{equation*}
		where
		\begin{equation*}
			I_3=-\frac{2\theta}{r+2}\int_{\mathbb{R}^n} u^{r+2}v^{s-1}\nabla v\cdot \nabla\varphi.
		\end{equation*}
	\end{lemma}
	
	\begin{proof}
		By  using Picone's identity (See Lemma \ref{picone-identity}), we have
		\begin{equation}\label{estimates-picone-3}
			\begin{aligned}
				&\int_{\mathbb{R}^n} u^rv^s|\nabla u|^2\varphi\geq\int_{\mathbb{R}^n} u^rv^s\nabla\big(\frac{u^2}{v^\theta}\big)\cdot\nabla v^\theta\varphi\\
				=&\theta\int_{\mathbb{R}^n}  v^{s+\theta+\frac{\theta r}{2}-1}\big(\frac{u^2}{v^\theta}\big)^{\frac{r}{2}}\nabla\big(\frac{u^2}{v^\theta}\big)\cdot\nabla v\varphi\\
				=&\frac{2\theta}{r+2}\int_{\mathbb{R}^n} v^{s+\theta+\frac{\theta r}{2}-1} \nabla\big(\frac{u^2}{v^\theta}\big)^{\frac{r+2}{2}}\cdot\nabla v \varphi\\
			\end{aligned}
		\end{equation}
		Using the divergence theorem and \eqref{lesys}, we have
		\begin{equation}\label{estimates-picone-4}
			\begin{aligned}
				&\frac{2\theta}{r+2}\int_{\mathbb{R}^n}  v^{s+\theta+\frac{\theta r}{2}-1}\nabla\big(\frac{u^2}{v^\theta}\big)^{\frac{r+2}{2}}\cdot\nabla v\varphi\\
				=&-\frac{2\theta}{r+2}\int_{\mathbb{R}^n} \big(\frac{u^2}{v^\theta}\big)^{\frac{r+2}{2}}\big(v^{s+\theta+\frac{\theta r}{2}-1}\Delta v+(\theta+s+\frac{\theta r}{2}-1)v^{s+\theta+\frac{\theta r}{2}-2}|\nabla v|^2\big)\varphi-I_3\\
				=&\frac{2\theta}{r+2}\int_{\mathbb{R}^n} u^{r+q+2}v^{s-1}\varphi-\frac{2\theta}{r+2}\big(\theta+s+\frac{\theta r}{2}-1\big)\int_{\mathbb{R}^n} u^{r+2}v^{s-2}|\nabla v|^2\varphi-I_3.
			\end{aligned}
		\end{equation}
		Combining  \eqref{estimates-picone-3} with \eqref{estimates-picone-4}, we finish the proof of this lemma.
	\end{proof}

	\begin{lemma}\label{estimates-combine}
		Suppose that $r,s\in\mathbb{R}$ with $r\not\in\{-1,-2\}.$ If $(u,v)$ is a pair of positive solutions to the system \eqref{lesys}, for any non-negative  $\varphi\in C_c^2(\mathbb{R}^n)$,  we have
		\begin{equation}\label{estimates-combine-1}
			\frac{1}{r+1}\int_{\mathbb{R}^n} u^{r+1}v^{p+s}\varphi+\mu_1\int_{\mathbb{R}^n} u^{r+2}v^{s-2}|\nabla v|^2\varphi\geq \mu_2\int_{\mathbb{R}^n} u^{r+q+2}v^{s-1}\varphi+I_4,
		\end{equation}
		where
		\begin{equation*}
			\mu_1=\frac{s(s-1)}{(r+1)(r+2)}+\frac{2\theta}{r+2}\big(\theta+s+\frac{\theta r}{2}-1\big),~~\mu_2=\frac{1}{r+2}\big(2\theta+\frac{s}{r+1}\big)
		\end{equation*}
		and
		\begin{equation*}
			I_4=I_3-I_2.
		\end{equation*}
	\end{lemma}
	
	\begin{proof}
		Using  Lemma \ref{estimates-tansform} and Lemma \ref{estimates-picone}, we have
		\begin{equation*}
			\begin{aligned}
				&\frac{2\theta}{r+2}\int_{\mathbb{R}^n} u^{r+q+2}v^{s-1}\varphi-\frac{2\theta}{r+2}\big(\theta+s+\frac{\theta r}{2}-1\big)\int_{\mathbb{R}^n} u^{r+2}v^{s-2}|\nabla v|^2\varphi+I_3\\
				\leq&\frac{1}{r+1}\int_{\mathbb{R}^n} u^{r+1}v^{p+s}\varphi-\frac{s}{(r+1)(r+2)}\int_{\mathbb{R}^n} u^{r+q+2}v^{s-1}\varphi\\
				&~~~~~~~~~~~~~~~~~~~~~~~~~~~~~~~~~~~+\frac{s(s-1)}{(r+1)(r+2)}\int_{\mathbb{R}^n} u^{r+2}v^{s-2}|\nabla v|^2\varphi+I_2,
			\end{aligned}
		\end{equation*}
		which implies \eqref{estimates-combine-1}.
	\end{proof}

	Choosing  $r=q-2$ and replacing $s$ by $s+1$ in Lemma \ref{estimates-combine}, we have the following proposition.
	\begin{proposition}\label{keyestimates}
		Suppose that $q>0$ with $q\neq 1$ and $s, \tilde{\theta}\in\mathbb{R}$. If $(u,v)$ is a pair of positive solutions to the system \eqref{lesys}, then for any non-negative  $\varphi\in C_c^2(\mathbb{R}^n)$,  we have
		\begin{equation*}
			\mu_3\int_{\mathbb{R}^n} u^qv^{s-1}|\nabla v|^2\varphi\geq -\frac{1}{q-1}\int_{\mathbb{R}^n} u^{q-1}v^{p+s+1}\varphi+\mu_4\int_{\mathbb{R}^n} u^{2q}v^s\varphi+I_5,
		\end{equation*}
		where
		\begin{equation*}
			\mu_3=\frac{(s+\tilde{\theta} q)\big(s+\tilde{\theta}(q-2)+1\big)+\tilde{\theta}(\tilde{\theta}-1)q}{q(q-1)},~~\mu_4=\frac{2\tilde{\theta}(q-1)+s+1}{q(q-1)}
		\end{equation*}
		and
		\begin{equation*}
			I_5=\frac{1}{q-1}\int_{\mathbb{R}^n} u^{q-1}v^{s+1}\nabla u\cdot \nabla\varphi-\frac{2\tilde{\theta}(q-1)+s+1}{q(q-1)}\int_{\mathbb{R}^n} u^qv^s\nabla v\cdot\nabla\varphi.
		\end{equation*}
	\end{proposition}
	
	\hspace{3em}
	
	As  for  $q=1$, by using  Lemma \ref{estimates-tansform},  we have the following estimate.
	\begin{proposition}\label{keyestimates-biharmonic}
		Suppose that $s\in\mathbb{R}$ with $s
		\not \in \{0,-1\}$, and $(u,v)$ is a pair of positive solutions to the  system \eqref{lesys} with $q=1$.Then for non-negative any $\varphi\in C_c^2(\mathbb{R}^n)$, we have
		\begin{equation*}
			\int_{\mathbb{R}^n} uv^{s-1}|\nabla v|^2\varphi=\frac{1}{s}\int_{\mathbb{R}^n} u^2v^s\varphi-\frac{1}{s(s+1)}\int_{\mathbb{R}^n}v^{p+1+s}\varphi+I_6,
		\end{equation*}
		where
		\begin{equation*}
			I_6=\frac{1}{s(s+1)}\int_{\mathbb{R}^n} v^{s+1}\nabla u\nabla\varphi-\frac{1}{s}\int_{\mathbb{R}^n} uv^s\nabla v\nabla\varphi.
		\end{equation*}
	\end{proposition}
	
	\hspace{3em}
	
	By synthesizing the results of Lemma \ref{obata-inequality}, Proposition \ref{keyestimates}, and Proposition \ref{keyestimates-biharmonic}, we obtain the following central estimate of this paper.
	
	\begin{theorem}\label{estimate-theorem}
		Suppose that $p, q>0$, $k_1, k_2, r, s, \theta_1, \theta_2\in \mathbb{R}$. If $(u,v)$ is a positive solution of system \eqref{lesys}, then for any non-negative $\varphi, \psi\in C_c^2(\mathbb{R}^n)$, we have
		\begin{equation}\label{estimates-theorem1}
			\left\{ \begin{aligned}
				&\alpha_1\int_{\mathbb{R}^n} u^{r-2}|\nabla u|^4\varphi+\beta_1\int_{\mathbb{R}^n} u^rv^{2p}\varphi\leq \gamma_1\int_{\mathbb{R}^n} u^{r+q+1}v^{p-1}\varphi+I_7,\\
				&\alpha_2\int_{\mathbb{R}^n} v^{s-2}|\nabla v|^4\psi+\beta_2\int_{\mathbb{R}^n} u^{2q}v^s\psi\leq \gamma_2\int_{\mathbb{R}^n} u^{q-1}v^{p+s+1}\psi+I_8,
			\end{aligned} \right.
		\end{equation}
		where
		\begin{equation}\label{alpha1 def}
			\alpha_1=-\frac{n-1}{n}k_1^2+k_1(r-1)-\frac{1}{2}r(r-1),
		\end{equation}
		\begin{equation}\label{alpha2 def}
			\alpha_2=-\frac{n-1}{n}k_2^2+k_2(s-1)-\frac{1}{2}s(s-1),
		\end{equation}
		
		\begin{equation*}
			\beta_1=(\frac{3}{2}r-\frac{n+2}{n}k_1)\frac{2\theta_1(p-1)+r+1}{\delta_1}-\frac{n-1}{n},
		\end{equation*}
		
		\begin{equation*}
			\beta_2=(\frac{3}{2}s-\frac{n+2}{n}k_2)\frac{2\theta_2(q-1)+s+1}{\delta_2}-\frac{n-1}{n},
		\end{equation*}
		
		\begin{equation*}
			\gamma_1=(\frac{3}{2}r-\frac{n+2}{n}k_1)\frac{p}{\delta_1},~~~~\gamma_2=(\frac{3}{2}s-\frac{n+2}{n}k_2)\frac{q}{\delta_2},
		\end{equation*}
		
		\begin{equation}\label{delta1}
			\delta_1=(r+\theta_1 p)\big(r+\theta_1(p-2)+1\big)+\theta_1(\theta_1-1)p,
		\end{equation}
		
		\begin{equation}\label{delta2}
			\delta_2=(s+\theta_2 q)\big(s+\theta_2(q-2)+1\big)+\theta_2(\theta_2-1)q,
		\end{equation}
		
		\begin{equation*}
			\begin{aligned}
				I_7&=\frac{1}{2}\int_{\mathbb{R}^n} u^r|\nabla u|^2\Delta\varphi+(r-k_1)\int_{\mathbb{R}^n} u^{r-1}|\nabla u|^2\nabla u\cdot\nabla\varphi
				\\
				&~~~~~~+\big[\frac{2\theta_1(p-1)+2r+2}{\delta_1}(\frac{3}{2}r-\frac{n+2}{n}k_1)-1\big]\int_{\mathbb{R}^n} u^rv^p\nabla u\cdot\nabla\varphi\\
				&~~~~~~-\frac{1}{\delta_1}(\frac{3}{2}r-\frac{n+2}{n}k_1)\int_{\mathbb{R}^n} u^{r+1}v^p\Delta \varphi
			\end{aligned}
		\end{equation*}
		and
		\begin{equation*}
			\begin{aligned}
				I_8&=\frac{1}{2}\int_{\mathbb{R}^n} v^s|\nabla v|^2\Delta\psi+(s-k_2)\int_{\mathbb{R}^n} v^{s-1}|\nabla v|^2\nabla v\cdot\nabla\psi
				\\
				&~~~~~~+\big[\frac{2\theta_2(q-1)+2s+2}{\delta_2}(\frac{3}{2}s-\frac{n+2}{n}k_2)-1\big]\int_{\mathbb{R}^n} u^qv^s\nabla v\cdot\nabla\psi\\
				&~~~~~~-\frac{1}{\delta_2}(\frac{3}{2}s-\frac{n+2}{n}k_2)\int_{\mathbb{R}^n} u^qv^{s+1}\Delta\psi.
			\end{aligned}
		\end{equation*}
		Here,  we suppose  that
		$$
		\begin{cases}
			r\not \in\{0 , -1\},  &p=1,\\
			(p-1)(\frac{3}{2}r-\frac{n+2}{n}k_1)\delta_1>0, & p\not=1
		\end{cases}
		$$
		and
		$$
		\begin{cases}
			s\not \in\{0 , -1\},  &q=1,\\
			(q-1)(\frac{3}{2}s-\frac{n+2}{n}k_2)\delta_2>0, & q\not=1.
		\end{cases}
		$$
	\end{theorem}

	\hspace{4em}
	
	It is not hard to construct a smooth cut-off function $\eta(x)$  on $\mathbb{R}^n$ satisfying
	$0\leq \eta(x)\leq 1$,
	\[
	\eta(x) =
	\begin{cases}
		1, & x \in B_1(0), \\
		0, & x \in \mathbb{R}^n \setminus B_2(0).
	\end{cases}
	\]
	$$|\nabla \eta|\leq 2, \quad  |\Delta\eta|\leq 2.$$

	Prior to addressing the terms $I_7$ and $I_8$ in Theorem \ref{estimate-theorem}, we establish the following preparatory estimate.
	
	\begin{lemma}\label{uniform-estimates-error}
		Suppose that $p\geq q\geq1$ and  $pq>1$. Then,  for any integer $m>\alpha+\beta+2$,  we have
		\begin{equation*}
			\int_{\mathbb{R}^n} u^{q+1}\eta^m\leq C\big(\int_{\mathbb{R}^n}v^{p+1} \eta^m+1\big),
		\end{equation*}
		where $C$ is a positive constant which depends  on $p$, $q$, $m$ and $n$.
	\end{lemma}
	\begin{proof}
		We may assume that $p>q$. Otherwise, by using Lemma \ref{component-comparison},  we have $u=v$ and hence obtain this lemma.
		
		Let $\varphi=\eta^m$ and a direct computation yields that
		\begin{equation}\label{uniform-estimates-error2}
			|\nabla\varphi|\leq 2m\varphi^{1-\frac{1}{m}}~~\mbox{and}~~|\Delta\varphi|\leq 4m^2\varphi^{1-\frac{2}{m}}~~\mbox{in}~\mathbb{R}^n.
		\end{equation}\par
		
		Since $(u,v)$ is a pair of solutions to the system \eqref{lesys}, we have
		\begin{equation}\label{uniform-estimates-error3}
			\begin{aligned}
				\int_{\mathbb{R}^n} u^{\frac{\beta}{\alpha}-1}|\nabla u|^2\varphi&=\frac{\alpha}{\beta}\int_{\mathbb{R}^n} \nabla u\cdot\nabla\big(u^{\frac{\beta}{\alpha}}\varphi\big)-\frac{\alpha}{\beta}\int_{\mathbb{R}^n}u^{\frac{\beta}{\alpha}} \nabla u\cdot\nabla\varphi\\
				&=\frac{\alpha}{\beta}\int_{\mathbb{R}^n}u^{\frac{\beta}{\alpha}}v^p\varphi+\frac{\alpha^2}{\beta(\alpha+\beta)}
				\int_{\mathbb{R}^n}u^{\frac{\beta+\alpha}{\alpha}}\Delta\varphi
			\end{aligned}
		\end{equation}
		where $\alpha,\beta$ come from \eqref{alphabeta}.
		Meanwhile, using integration by parts, one has
		\begin{equation}\label{uniform-estimates-error4}
			\begin{aligned}
				\int_{\mathbb{R}^n} u^{q+1}\varphi&=\int_{\mathbb{R}^n} \nabla v\cdot\nabla u\varphi+\int_{\mathbb{R}^n}u \nabla v\cdot\nabla\varphi\\
				&=\int_{\mathbb{R}^n} v^{p+1}\varphi+\int_{\mathbb{R}^N}u \nabla v\cdot\nabla\varphi-\int_{\mathbb{R}^n}v \nabla u\cdot\nabla\varphi\\
				&=\int_{\mathbb{R}^n} v^{p+1}\varphi+\int_{\mathbb{R}^n}\nabla (uv)\cdot\nabla\varphi-2\int_{\mathbb{R}^n}v\nabla u\cdot\nabla\varphi\\
				&=\int_{\mathbb{R}^n} v^{p+1}\varphi-\int_{\mathbb{R}^n}uv\Delta\varphi-2\int_{\mathbb{R}^n}v\nabla u\cdot\nabla\varphi.
			\end{aligned}
		\end{equation}
		Therefore, by using \eqref{uniform-estimates-error2}, \eqref{uniform-estimates-error3}, \eqref{uniform-estimates-error4}, Lemma \ref{component-comparison}, H\"{o}lder's  inequality and Young's inequality,  we have
		\begin{equation*}
			\begin{aligned}
				\int_{\mathbb{R}^n} u^{\frac{\beta}{\alpha}-1}|\nabla u|^2\varphi&\leq\frac{\alpha}{\beta}\int_{\mathbb{R}^n}u^{\frac{q+1}{p+1}}v^p\varphi+\frac{\alpha^2C_1}{\beta(\alpha+\beta)}
				\int_{\mathbb{R}^n}u^{\frac{\beta+\alpha}{\alpha}}\varphi^{1-\frac{2}{m}}\\
				&\leq\frac{\alpha}{\beta}\big(\int_{\mathbb{R}^n} u^{q+1}\varphi\big)^{\frac{1}{p+1}}\big(\int_{\mathbb{R}^n} v^{p+1}\varphi\big)^{\frac{p}{p+1}}\\
				&~~~~~~~~~+\frac{\alpha^2C_1}{\beta(\alpha+\beta)}\big(\int_{\mathbb{R}^n} u^{q+1}\varphi\big)^{\frac{\alpha+\beta}{\alpha+\beta+2}}\big(\int_{\mathbb{R}^n} \varphi^{1-\frac{\alpha+\beta+2}{m}}\big)^{\frac{2}{\alpha+\beta+2}}\\
				&\leq \frac{1}{2}\int_{\mathbb{R}^n} u^{q+1}\varphi+C_2\big(\int_{\mathbb{R}^n} v^{p+1}\varphi+1\big)
			\end{aligned}
		\end{equation*}
		and
		\begin{equation*}
			\begin{aligned}
				\int_{\mathbb{R}^n} u^{q+1}\varphi&\leq\int_{\mathbb{R}^n} v^{p+1}\varphi+C_1\int_{\mathbb{R}^n}uv\varphi^{1-\frac{2}{m}}+C_2\int_{\mathbb{R}^n}u^{\frac{1}{2}(1+\frac{\beta}{\alpha})}u^{\frac{1}{2}(\frac{\beta}{\alpha}-1)} |\nabla u| \varphi^{1-\frac{1}{m}}\\
				&\leq\int_{\mathbb{R}^n} v^{p+1}\varphi+\big(\int_{\mathbb{R}^n} u^{q+1}\varphi\big)^{\frac{1}{q+1}}\big(\int_{\mathbb{R}^n} v^{p+1}\varphi\big)^{\frac{1}{p+1}}\big(\int_{\mathbb{R}^n} \varphi^{1-\frac{\alpha+\beta+2}{m}}\big)^{\frac{2}{\alpha+\beta+2}}\\
				&~~+C_2\big(\int_{\mathbb{R}^n} u^{q+1}\varphi\big)^{\frac{\alpha+\beta}{2(\alpha+\beta+2)}}\big(\int_{\mathbb{R}^n} u^{\frac{\beta}{\alpha}-1}|\nabla u|^2\varphi\big)^{\frac{1}{2}}\big(\int_{\mathbb{R}^n} \varphi^{1-\frac{\alpha+\beta+2}{m}}\big)^{\frac{1}{\alpha+\beta+2}}\\
				&\leq \frac{1}{4}\int_{\mathbb{R}^n} u^{q+1}\varphi+\frac{1}{2}\int_{\mathbb{R}^n}u^{\frac{\beta}{\alpha}-1}|\nabla u|^2\varphi+C_3\big(\int_{\mathbb{R}^n} v^{p+1}\varphi+1\big)\\
				&\leq\frac{1}{2}\int_{\mathbb{R}^n} u^{q+1}\varphi+C_4\big(\int_{\mathbb{R}^n} v^{p+1}\varphi+1\big),
			\end{aligned}
		\end{equation*}
		where $C_1$, $C_2$, $C_3$ and $C_4$ are positive constants depending only on $p$, $q$, $m$ and $n$.
		
		Thus, we finish the proof of  this lemma.
	\end{proof}

	Now, we are ready to estimate  $I_7$ and $I_8$. From now on, we assume that $(u,v)$ is a pair of positive solutions to the  system \eqref{lesys}. In order to balance the two equations in \eqref{estimates-theorem1}, we assume that
	$r$ and $s$ satisfies the following crucial  relationship
	\begin{equation}\label{r-s-relation}
		\frac{r+2}{q+1}=\frac{s+2}{p+1}.
	\end{equation}
	By  a direct calculation and \eqref{alphabeta}, we deduce that relation \eqref{r-s-relation} is equivalent to
	\begin{equation*}
		(r+2)\alpha=(s+2)\beta
	\end{equation*}
	when $pq\not=1$.
	\begin{lemma}\label{estimates-error}
		Suppose that $p\geq q\geq1$ with $pq>1$.  If $\varphi=\psi=\eta^m$ with $m\in\mathbb{N}^+$ and \eqref{r-s-relation} holds with $-2< r\leq q-1$ and  $-2<s\leq p-1$, then there exits positive constant $m_0$ depending on on $p$, $q$ and $r$ such that for any $m>m_0$ and any $\varepsilon>0$ we have
		\begin{equation}\label{estimates-error1}
			\begin{aligned}
				|I_7|\leq &\varepsilon\int_{\mathbb{R}^n}\big(u^{r-2}|\nabla u|^4+u^{2q}v^s\big)\varphi+C\\
				|I_8|\leq & \varepsilon\int_{\mathbb{R}^n}\big( v^{s-2}|\nabla v|^4+u^{2q}v^s\big)\varphi+C
			\end{aligned}
		\end{equation}
		where $C$ is  a positive constant which depends on $p$, $q$, $k_1$, $k_2$, $\theta_1$, $\theta_2$, $r$, $s$,  $m$, $\varepsilon$ and $n$.
	\end{lemma}
	\begin{proof}
		During  the proof of this lemma, if a positive constant depends only on $p$, $q$, $k_1$, $k_2$, $r$, $m$, $\varepsilon$ and $n$, we will denote it by $C$ for brevity.

		For any $-2<s\leq p-1$, by using Lemma \ref{component-comparison}, one has
		\begin{equation*}
			v^{s+2}\leq C (v^su^{2q})^{\vartheta_1}
		\end{equation*}
		where 	$\vartheta_1=\frac{(r+2)}{(r+2)+\frac{4}{\alpha}}\in(0,1).$
		As a direct consequence, one has
		\begin{equation}\label{estimates-error3}
			\int_{\mathbb{R}^n}v^{s+2}\eta^m\leq C\int_{\mathbb{R}^n} (v^su^{2q})^{\vartheta_1}\eta^m+C,
		\end{equation}
		
		Now, we claim that 	
		\begin{equation}\label{estimates-error2}
			\int_{\mathbb{R}^n}u^{r+2}\eta^m\leq C\int_{\mathbb{R}^n} (v^su^{2q})^{\vartheta_2}\eta^m+C,
		\end{equation}
		for some $\vartheta_2\in (0,1)$.
		
		Firstly, when $s\in(-2,0]$, by using \eqref{r-s-relation} and Lemma \ref{component-comparison},  we have
		\begin{equation*}
			u^{r+2}\leq C (v^su^{2q})^{\vartheta_1}
		\end{equation*}
		which deduces \eqref{estimates-error2} by choosing $\vartheta_2=\vartheta_1$.
		Secondly, if $0< s\leq p-1$, the situation is different.  By using \eqref{r-s-relation}, one has
		\begin{equation}\label{r<q-1}
			-\frac{2(p-q)}{p+1}<r\leq q-1.
		\end{equation}
		Then, 	one may easily check that  $\frac{q+1}{r+2+\frac{4}{\alpha}}\in(0,1)$ by using \eqref{r<q-1}.
		With the help of Lemma \ref{component-comparison}, we deduce that
		\begin{equation*}
			u^{r+2}=u^{\nu_1(q+1)}, \quad v^{p+1}\leq C (v^su^{2q})^{\nu_2}~~\mbox{in}~\mathbb{R}^n.
		\end{equation*}
		where $\nu_1=\frac{r+2}{q+1}\in(0,1]$ and  $\nu_2=\frac{q+1}{r+2+\frac{4}{\alpha}}\in(0,1)$.
		By using Lemma \ref{uniform-estimates-error} and Young's  inequality, we have
		\begin{equation*}
			\begin{aligned}
				\int_{\mathbb{R}^n}u^{r+2}\eta^m&=\int_{\mathbb{R}^n} u^{\nu_1(q+1)}\eta^m\leq C\int_{\mathbb{R}^n} u^{q+1}\eta^m+C\\
				&\leq C\int_{\mathbb{R}^n} v^{p+1}\eta^m+C\leq C\int_{\mathbb{R}^n} (v^su^{2q})^{\nu_2}\eta^m+C,
			\end{aligned}
		\end{equation*}
		which implies \eqref{estimates-error2} by choosing $\vartheta_2=\nu_2.$

		Now, we turn  to  prove \eqref{estimates-error1}.
		By using  Lemma \ref{component-comparison} and \eqref{r-s-relation}, we have
		\begin{equation*}
			u^{r+\frac{2}{3}}v^{\frac{4p}{3}}\leq C (v^su^{2q})^{\overline{\vartheta}}~~\mbox{and}~~u^{\frac{4q}{3}}v^{s+\frac{2}{3}}\leq C (v^su^{2q})^{\overline{\vartheta}}~~\mbox{in}~\mathbb{R}^n
		\end{equation*}
		with $\overline{\vartheta}=\frac{(r+2)\alpha+\frac{8}{3}}{(r+2)\alpha+4}\in(0,1)$.
		Hence,  by \eqref{uniform-estimates-error2}, \eqref{estimates-error2}, \eqref{estimates-error3}, Young's inequality and H\"older's inequality, we have
		\begin{equation}\label{estimates-error8}
			\begin{aligned}
				|\int_{\mathbb{R}^n} u^r|\nabla u|^2\Delta\varphi|&\leq C\int_{\mathbb{R}^n} u^r|\nabla u|^2\varphi^{1-\frac{2}{m}}\\
				&\leq C\int_{\mathbb{R}^n} \big(u^{r-2}|\nabla u|^4\varphi\big)^{\frac{1}{2}}\big(u^{r+2}\varphi^{1-\frac{4}{m}}\big)^{\frac{1}{2}}\\
				&\leq\varepsilon\int_{\mathbb{R}^n}u^{r-2}|\nabla u|^4 \varphi+C\int_{\mathbb{R}^n}u^{r+2}\eta^{m-4}\\
				&\leq\varepsilon\int_{\mathbb{R}^n}u^{r-2}|\nabla u|^4 \varphi+ C\int_{\mathbb{R}^n} (v^su^{2q})^{\vartheta_2}\eta^{m-4}+C\\
				&=\varepsilon\int_{\mathbb{R}^n}u^{r-2}|\nabla u|^4 \varphi+ C\int_{\mathbb{R}^n} (v^su^{2q}\varphi)^{\vartheta_2}\varphi^{1-\vartheta_2-\frac{4}{m}}+C\\
				&\leq \varepsilon\int_{\mathbb{R}^n}u^{r-2}|\nabla u|^4 \varphi+ \varepsilon\int_{\mathbb{R}^n}v^su^{2q}\varphi+\int_{\mathbb{R}^n} \varphi^{1-\frac{4}{m(1-\vartheta_2)}}+C\\
				&\leq\varepsilon\int_{\mathbb{R}^n}u^{r-2}|\nabla u|^4 \varphi+ \varepsilon\int_{\mathbb{R}^n}v^su^{2q}\varphi+C,
			\end{aligned}
		\end{equation}

		\begin{equation}\label{estimates-error9}
			\begin{aligned}
				|\int_{\mathbb{R}^n} u^{r-1}|\nabla u|^2\nabla u\nabla\varphi|&\leq C\int_{\mathbb{R}^N} u^{r-1}|\nabla u|^3\varphi^{1-\frac{1}{m}}\\
				&\leq C\int_{\mathbb{R}^n} \big(u^{r-2}|\nabla u|^4\varphi\big)^{\frac{3}{4}}\big(u^{r+2}\varphi^{1-\frac{4}{m}}\big)^{\frac{1}{4}}\\
				&\leq\varepsilon\int_{\mathbb{R}^n}u^{r-2}|\nabla u|^4 \varphi+C\int_{\mathbb{R}^n}u^{r+2}\eta^{m-4}\\
				&\leq\varepsilon\int_{\mathbb{R}^n}u^{r-2}|\nabla u|^4 \varphi+ \varepsilon\int_{\mathbb{R}^n}v^su^{2q}\varphi+\int_{\mathbb{R}^N} \varphi^{1-\frac{4}{m(1-\vartheta_2)}}+C\\
				&\leq\varepsilon\int_{\mathbb{R}^n}u^{r-2}|\nabla u|^4 \varphi+ \varepsilon\int_{\mathbb{R}^n}v^su^{2q}\varphi+C,
			\end{aligned}
		\end{equation}

		\begin{equation}\label{estimates-error10}
			\begin{aligned}
				|\int_{\mathbb{R}^n} u^rv^p\nabla u\nabla\varphi|&\leq C\int_{\mathbb{R}^n} u^rv^p|\nabla u|\varphi^{1-\frac{1}{m}}\\
				&\leq C\int_{\mathbb{R}^n} \big(u^{r-2}|\nabla u|^4\varphi\big)^{\frac{1}{4}}\big(v^su^{2q}\varphi\big)^{\frac{3\overline{\vartheta}}{4}}\varphi^{1-\frac{1}{m}-\frac{1}{4}-\frac{3\overline{\vartheta}}{4}}\\
				&\leq C\big(\int_{\mathbb{R}^n} u^{r-2}|\nabla u|^4\varphi\big)^{\frac{1}{4}}\big(\int_{\mathbb{R}^n} v^su^{2q}\varphi\big)^{\frac{3\overline{\vartheta}}{4}}\big(\int_{\mathbb{R}^n}\varphi^{1-\frac{4}{3m(1-\overline{\vartheta})}}\big)^{\frac{3(1-\overline{\vartheta})}{4}}\\
				&\leq\varepsilon\int_{\mathbb{R}^n}v^{s-2}|\nabla v|^4 \varphi+\varepsilon\int_{\mathbb{R}^n}u^{2q}v^s\varphi+C
			\end{aligned}
		\end{equation}
		
		and
		\begin{equation}\label{estimates-error11}
			\begin{aligned}
				|\int_{\mathbb{R}^n} u^{r+1}v^p\Delta\varphi|&\leq C\int_{\mathbb{R}^n} u^{\frac{r+2}{2}}\big(u^{\frac{r}{2}}v^p\varphi^{\frac{1}{2}}\big)\varphi^{\frac{1}{2}-\frac{2}{m}}\\
				&=\int_{\mathbb{R}^n} \big(u^rv^{2p}\big)^{\frac{1}{2}}\big(u^{r+2}\varphi^{1-\frac{4}{m}}\big)^{\frac{1}{2}}\\
				&\leq\varepsilon\int_{\mathbb{R}^n}u^{r-2}|\nabla u|^4 \varphi+ \varepsilon\int_{\mathbb{R}^n}v^su^{2q}\varphi+\int_{\mathbb{R}^N} \varphi^{1-\frac{4}{m(1-\vartheta_2)}}+C\\
				&\leq\varepsilon\int_{\mathbb{R}^n}u^{r-2}|\nabla u|^4 \varphi+ \varepsilon\int_{\mathbb{R}^n}v^su^{2q}\varphi+C,
			\end{aligned}
		\end{equation}
		provided $m>m_0=\max\{\alpha+\beta+6, \frac{4}{1-\vartheta_2},\frac{4}{3(1-\overline{\vartheta})}\}>0$. Therefore, by \eqref{estimates-error8}-\eqref{estimates-error11}, we have
		\begin{equation*}
			|I_7|\leq\varepsilon\int_{\mathbb{R}^n}\big(u^{r-2}|\nabla u|^4+u^{2q}v^s\big)\varphi+C.
		\end{equation*}
		
		Similarly, we have
		\begin{equation*}
			|I_8|\leq\varepsilon\int_{\mathbb{R}^n}\big(v^{s-2}|\nabla v|^4+u^{2q}v^{s}\big)\varphi+C.
		\end{equation*}
		Thus, we finish the proof of this lemma.
	\end{proof}

	\section{Choice of coefficients}\label{sec:choice}

	Our current objective is to select appropriate values for $r,s,k_1,k_2,\theta_1,\theta_2$  in  Theorem \ref{estimate-theorem} with the aim of ensuring the following positivity conditions hold:
	\begin{equation}\label{beta1beta2>gamma1gamma2}
		\alpha_1>0,\; \alpha_2>0, \;\beta_1>0, \; \beta_2>0,\; \gamma_1>0, \; \gamma_2>0,\; \beta_1\beta_2-\gamma_1\gamma_2>0.
	\end{equation}

	To serve this aim,  $p$ and $q$ are parameterized by $d_1$ and $d_2$ through the following system:
	\begin{equation}\label{pqd1d2}
		\begin{cases}
			\dfrac{1}{p+1} = \dfrac{1}{2} - \dfrac{d_1}{n}, \\[8pt]
			\dfrac{1}{q+1} = \dfrac{1}{2} - \dfrac{d_2}{n}.
		\end{cases}
	\end{equation}
	Then the restrictions  $p\geq 1, q\geq 1$ and \eqref{subcri} is  equivalent to
	\begin{equation*}
		d_1\geq 0, \quad d_2\geq 0,\quad  d_1+d_2<2.
	\end{equation*}

	When $d_1=d_2$, Theorem \ref{mainthm} holds by using 	Lemma \ref{component-comparison} and the classical result of Gidas and Spruck in \cite{GS}. Hence, we suppose that $$d_1>d_2$$ without loss of generality.
	
	Choose
	\begin{equation}\label{rs k1 k2choice}
		\left\{
		\begin{aligned}
			r=&-\frac{3(d_1-d_2)}{n-2d_2}, \quad s=\frac{d_1-d_2}{n-2d_1}\\
			k_1=&\frac{n}{2(n-1)}\left(r-2+\frac{r}{n}+\frac{r^2}{2}\right),\\
			k_2=&\frac{n}{2(n-1)}\left(s-2+\frac{s}{n}+\frac{s^2}{2}\right)\\
			\theta_1=&\frac{\mu(r)-r}{p}, \quad \theta_2=\frac{\mu(s)-s}{q}
		\end{aligned}\right.
	\end{equation}
	where   the function $\mu(x)$ is given by
	$$\mu(x):=\frac{(n+2)n+(n^2-3n-1)x-\frac{n(n+2)}{4}x^2}{(n-1)^2}.$$

	\begin{lemma}\label{lem:crucial estimate}
		Choosing  \eqref{rs k1 k2choice} in Theorem \ref{estimate-theorem}, for any constant $c_0>2$, there exists  a positive  integer $N(c_0)$ such that, for any  $n\geq N(c_0)$ and
		\begin{equation}\label{d1d2range}
			0\leq d_2<d_1\leq 2-d_2-\frac{c_0}{n},
		\end{equation}
		then the positivity condition  \eqref{beta1beta2>gamma1gamma2} holds.
	\end{lemma}
	\begin{proof}
		
		Based on the choice of $r,s,k_1,k_2$ in \eqref{rs k1 k2choice}, one has
		\begin{equation*}
			\frac{n\left(\frac{3}{2}r-\frac{n+2}{n}k_1\right)}{(n-1)}=\mu(r), \quad \frac{n\left(\frac{3}{2}s-\frac{n+2}{n}k_2\right)}{(n-1)}=\mu(s).
		\end{equation*}
		Under  the restriction \eqref{d1d2range} and  the choice \eqref{rs k1 k2choice},
		one has
		\begin{equation}\label{rs-range1}
			-\frac{6}{n}<-\frac{3d_1}{n}\leq r<0, \quad 0<s\leq \frac{d_1}{n-2d_1}<\frac{2}{n-4}
		\end{equation}
		and then, for $n\geq 6$, there holds
		\begin{align*}
			\mu(r)&>\mu\left(-\frac{6}{n}\right)=\frac{n^3-4n^2+9n-12}{n(n-1)^2}>0,\\
			\mu(s)&>\mu(0)=\frac{n(n+2)}{(n-1)^2}>0.
		\end{align*}
		Thus, under the constraint \eqref{d1d2range} and $n\geq 6$, there holds
		\begin{equation}\label{mu(r)>0,mu(s)>0}
			\mu(r)>0,\quad \mu(s)>0.
		\end{equation}		
		Inserting the choice  $\theta_1,\theta_2$ from \eqref{rs k1 k2choice} into \eqref{delta1} and  \eqref{delta2}, one has
		\begin{equation}\label{delata1,delata2}
			\begin{aligned}
				\delta_1=&\frac{r^2}{p}+r+\frac{p-1}{p}\mu(r)^2,\\
				\delta_2=&\frac{s^2}{q}+s+\frac{q-1}{q}\mu(s)^2.
			\end{aligned}
		\end{equation}
		Then, by  using \eqref{delata1,delata2}, \eqref{rs-range1} and the condition $q\geq 1$, it is easy to see that  $$\delta_2>0.$$ To show $\delta_1>0$, we need some additional efforts. Firstly, for $n\geq 13$ and $0<d_1<2$,  one has
		$$-\frac{n+2d_1}{2(n-2d_1)}<-\frac{3d_1}{n}.$$
		Then,  using \eqref{rs-range1}, we have
		$$\frac{r^2}{p}+r> -\frac{3d_1}{n}+\frac{n-2d_1}{n+2d_1}\frac{9d_1^2}{n^2}=-\frac{3d_1(n^2-d_1n+6d_1^2)}{n^2(n+2d_1)}.$$
		Combining these estimates, one has
		\begin{align*}
			\delta_1=&\frac{r^2}{p}+r+\frac{p-1}{p}\mu(r)^2\\
			\geq &-\frac{3d_1(n^2-d_1n+6d_1^2)}{n^2(n+2d_1)}+\frac{4d_1}{n+2d_1}\mu\left(-\frac{6}{n}\right)^2\\
			=&\frac{d_1}{n+2d_1}\left(-3+\frac{3d_1(n-6d_1)}{n^2}+4\mu\left(-\frac{6}{n}\right)^2\right)
		\end{align*}
		Then, for $n\geq 13$ and $0<d_1<2$,
		a direct computation with the help of Mathematica yields that
		\begin{equation*}
			\mu\left(-\frac{6}{n}\right)^2>\frac{3}{4}, \quad n-6d_1>0.
		\end{equation*}
		Thus,  for $n\geq 13$, there holds
		$$\delta_1>0.$$
		Combining  with \eqref{mu(r)>0,mu(s)>0}, for $n\geq 13$, there holds
		\begin{equation}\label{gamma1>0}
			\gamma_1>0,\quad \gamma_2>0.
		\end{equation}
		
		For brevity, we	set the following  two notations
		\begin{equation}\label{A1 def}
			A_1:=\frac{p\beta_1}{\gamma_1}=2\theta_1(p-1)+r+1-\frac{n-1}{n}\frac{\delta_1}{\frac{3}{2}r-\frac{n+2}{n}k_1}
		\end{equation}
		and
		\begin{equation}\label{A2 def}
			A_2:=\frac{q\beta_2}{\gamma_2}=2\theta_2(q-1)+s+1-\frac{n-1}{n}\frac{\delta_2}{\frac{3}{2}s-\frac{n+2}{n}k_2}.
		\end{equation}
		To show $\beta_1>0$ and $\beta_2>0$ for $n\geq 13$, it is sufficient to show $A_1>0$ and $A_2>0$  with the help of \eqref{gamma1>0}.
		
		Inserting  $k_1,k_2,\theta_1,\theta_2$ from \eqref{rs k1 k2choice} into \eqref{A1 def} and \eqref{A2 def}, one has
		\begin{equation}\label{A1A2repre}
			\begin{aligned}
				A_1=&\left(1-\frac{r}{\mu(r)}\right)\left(\frac{p-1}{p}\mu(r)+\frac{r}{p}+1\right),\\
				A_2=&\left(1-\frac{s}{\mu(s)}\right)\left(\frac{q-1}{q}\mu(s)+\frac{s}{q}+1\right).
			\end{aligned}
		\end{equation}
		From \eqref{rs-range1} and \eqref{mu(r)>0,mu(s)>0},  for $n\geq 13$ and $p\geq 1$, there holds
		$$1-\frac{r}{\mu(r)}>0, \quad \frac{p-1}{p}\mu(r)+\frac{r}{p}+1\geq \frac{r}{p}+1>0$$			
		which yields that
		\begin{equation}\label{A_1>0}
			A_1>0.
		\end{equation}
		By using \eqref{rs-range1},  for $q\geq 1$ and $n\geq 13$, we have
		$$\frac{q-1}{q}\mu(s)+\frac{s}{q}+1>0, \quad 1-\frac{s}{\mu(s)}>1-\frac{2}{n-4}\cdot\frac{(n-1)^2}{n(n+2)}>0$$
		which yields that
		\begin{equation}\label{A_2>0}
			A_2>0.
		\end{equation}
		Combining \eqref{A_1>0}, \eqref{A_2>0} with \eqref{gamma1>0},  for $n\geq 13$, there holds
		$$\beta_1>0, \quad \beta_2>0.$$

		For $n\geq 13$, to show $\alpha_1>0$ and $\alpha_2>0$ in \eqref{alpha1 def} and \eqref{alpha2 def},  it is equivalent to show that
		\begin{equation}\label{k1range}
			-\sqrt{1-\frac{2r}{n}-\frac{(n-2)r^2}{n}}<\frac{2(n-1)}{n}k_1-r+1<\sqrt{1-\frac{2r}{n}-\frac{(n-2)r^2}{n}}
		\end{equation}
		and
		\begin{equation}\label{k2range}
			-\sqrt{1-\frac{2s}{n}-\frac{(n-2)s^2}{n}}<\frac{2(n-1)}{n}k_2-s+1<\sqrt{1-\frac{2s}{n}-\frac{(n-2)s^2}{n}}.
		\end{equation}
		Based on our choice \eqref{rs k1 k2choice}, the inequalities \eqref{k1range} and \eqref{k2range} can be transformed to
		\begin{equation}\label{k1range2}
			-\sqrt{1-\frac{2r}{n}-\frac{(n-2)r^2}{n}}<-1+\frac{r}{n}+\frac{r^2}{2}<\sqrt{1-\frac{2r}{n}-\frac{(n-2)r^2}{n}},
		\end{equation}
		and
		\begin{equation}\label{k2range2}
			-\sqrt{1-\frac{2s}{n}-\frac{(n-2)s^2}{n}}<-1+\frac{s}{n}+\frac{s^2}{2}<\sqrt{1-\frac{2s}{n}-\frac{(n-2)s^2}{n}}.
		\end{equation}
		For $n\geq 13$, using \eqref{rs-range1}, one can easily show that
		$$-1+\frac{r}{n}+\frac{r^2}{2}\leq 0,\quad -1+\frac{s}{n}+\frac{s^2}{2}\leq 0.$$
		Thus, for $n\geq 13$, to show \eqref{k1range2} and \eqref{k2range2}, we only need to show that
		\begin{align*}
			\left(1-\frac{2r}{n}-\frac{n-2}{n}r^2\right)-\left(-1+\frac{r}{n}+\frac{r^2}{2}\right)^2=&\frac{r^2(8n-4-4nr-n^2r^2)}{4n^2}>0,\\
			\left(1-\frac{2s}{n}-\frac{n-2}{n}s^2\right)-\left(-1+\frac{s}{n}+\frac{s^2}{2}\right)^2=&\frac{s^2(8n-4-4ns-n^2s^2)}{4n^2}>0.
		\end{align*}
		By using \eqref{rs-range1}, it is not hard to show that the above estimates holds when $n\geq 13$ via direct computations. Thus, for $n\geq 13$, there holds
		$$\alpha_1>0, \quad \alpha_2>0.$$

		From \eqref{A1 def} and \eqref{A2 def}, to show $\beta_1\beta_2>\gamma_1\gamma_2$ for large $n$, it is sufficient to show that
		\begin{equation*}
			A_1A_2-pq>0.
		\end{equation*}

		To address this situation, we partition the range \eqref{d1d2range} into two cases:
		\begin{enumerate}
			\item
			\begin{equation*}
				0\leq d_2<d_1\leq \frac{1}{2},
			\end{equation*}
			\item \begin{equation*}
				0\leq d_2<1,\quad \max\{d_2,\frac{1}{2}\}<d_1<2-d_2-\frac{c_0}{n}.
			\end{equation*}
		\end{enumerate}

		\hspace{3em}
		
		{\bf Case (1):  $	0\leq d_2<d_1\leq \frac{1}{2}.$}
		
		\hspace{3em}
		
		Using \eqref{A1A2repre},   a direct computation yields that
		\begin{align*}
			A_1-p=&\frac{\left(\mu(r)-p-r\right)\left((p-1)\mu(r)+r\right)}{\mu(r)p}\\
			A_2-q=&\frac{\left(\mu(s)-q-s\right)\left((q-1)\mu(s)+s\right)}{\mu(s)q}
		\end{align*}
		Set
		$$\xi(x):=\mu(x)-x=\frac{(n+2)(-nx^2-4x+4n)}{4(n-1)^2}$$
		
		When $0\leq d_2<d_2\leq\frac{1}{2}$, from \eqref{rs k1 k2choice},  one has
		$$-\frac{3}{2n}<r<0,\quad 0<s\leq \frac{1}{n-1}, \quad 1\leq p\leq \frac{n+1}{n-1}, \quad 1\leq q\leq \frac{n+1}{n-1}.$$
		Then,  for $n\geq 13$, one has
		$$
		\mu(r)-r-p=\xi(r)-p\geq \xi(0)-\frac{n+1}{n-1}=\frac{2n+1}{(n-1)^2}>0,
		$$
		and
		$$\mu(s)-s-q=\xi(s)-q\geq \xi\left(\frac{1}{n-1}\right)-\frac{n+1}{n-1}=\frac{8n^3-17n^2-6n+12}{4(n-1)^4}>0.$$
		For $n\geq 13$, a direct computation yields that
		$$\mu(r)\geq \mu\left(-\frac{3}{2n}\right)=\frac{16n^3+8n^2+63n+6}{16(n-1)^2n}>1.$$
		$$\mu(s)\geq\mu(0)=\frac{n(n+2)}{(n-1)^2}>0.$$
		Using the above  estimates,  for $n\geq 13$, one has
		$$\mu(r)(p-1)+r=\mu(r)\frac{4d_1}{n-2d_1}-\frac{3(d_1-d_2)}{n-2d_2}>\frac{d_1}{n}>0,$$
		and
		$$\mu(s)(q-1)+s>0.$$
		Thus, for $n\geq 13 $ and $0\leq d_2<d_1\leq \frac{1}{2}$, one has
		$$A_1>p,\quad A_2>q$$
		which yields that
		$$A_1A_2>pq.$$
		
		\hspace{3em}
		
		{\bf Case (2):  $0\leq d_2<1,\max\left\{d_2,\frac{1}{2}\right\}<d_1<2-d_2-\frac{c_0}{n}.$}
		
		\hspace{3em}
		
		With the help of the choice \eqref{rs k1 k2choice} and  Mathematica, we  obtain that
		\begin{align*}
			&A_1A_2-pq\\
			=&\frac{4(2-d_1-d_2)(d_1+3d_2)}{n^2}+\frac{P(d_1,d_2)}{n^3}\\
			&+\frac{\sum^{17}_{i=0}Q_i(d_1,d_2)n^i}{2(n-1)^4n^3(n-2d_1)^2(n+2d_1)(n-2d_2)^2(n+2d_2)A_3A_4}
		\end{align*}
		where
		\begin{equation}\label{A3def}
			\begin{aligned}
				A_3=&4n^4+4(2-3d_1-d_2)n^3\\
				&+(-44d_1+7d_1^2+12d_2+10d_1d_2-d_2^2)n^2\\
				&+(-4d_1+54d_1^2+4d_2-20d_1d_2-2d_2^2)n\\
				&+8d_1(d_1-d_2),
			\end{aligned}
		\end{equation}		
		\begin{equation}\label{A4def}
			\begin{aligned}
				A_4=&4n^4+4(2-3d_1-d_2)n^3\\
				&+(36d_1-9d_1^2-68d_2+42d_1d_2-17d_2^2)n^2\\
				&+(12d_1-18d_1^2-12d_2-36d_1d_2+86d_2^2)n\\
				&+24d_2(d_2-d_1)
			\end{aligned}
		\end{equation}
		and $E_i(d_1,d_2),F_i(d_1,d_2), P(d_1,d_2)$, $Q_i(d_1,d_2)$ are polynomials
		in $d_1,d_2$ with integer coefficients.
		The precise representation of $P(d_1,d_2)$ are given by
		\begin{align*}
			P(d_1,d_2)=&\frac{1}{2}\left(19d_{1}^{3} - d_{1}^{2}(92+125d_2)\right)\\
			&+\frac{1}{2}d_1(92+232d_2-127d_2^2)\\
			&+\frac{1}{2}d_2(20-12d_2-23d_2^2).
		\end{align*}
		which will be used later. More precise expressions for other polynomials are available but are too lengthy to include here; they can be generated using Mathematica if needed. To maintain conciseness, we omit them, noting that each is bounded by universal constants under the constraint $$0\leq d_2<d_1<2-d_2.$$

		Then, there exist a large  integer $n_1\geq 13$ such that for any $n\geq n_1$, there holds
		$$ 2(n-1)^4(n-2d_1)^2(n+2d_1)(n-2d_2)^2(n+2d_2)A_3A_4\geq 20n^{18}.$$
		One can also find  a positive constant $C_0$ such that, for $n\geq13$,
		$$\sum^{17}_{i=0}Q_i(d_1,d_2)n^i\geq -C_0\cdot n^{17}.$$
		
		Thus,  for $n\geq n_1$, one has
		$$\frac{\sum^{17}_{i=0}Q_i(d_1,d_2)n^i}{2(n-1)^4n^3(n-2d_1)^2(n+2d_1)(n-2d_2)^2(n+2d_2)A_3A_4}\geq -\frac{C_0}{20n^4}.$$
		With the help of Mathematica, one can show that
		$$P(d_1,d_2)+8(d_1+3d_2)\geq 0$$
		under the constraint
		$$0\leq d_2<d_1<2-d_2.$$
		Thus, for $n\geq n_1$ and $0\leq d_2<1,\max\left\{d_2,\frac{1}{2}\right\}<d_1<2-d_2-\frac{c_0}{n}$, there holds
		\begin{align*}
			A_1A_2-pq\geq &\frac{4(2-d_1-d_2)(d_1+3d_2)}{n^2}-\frac{8(d_1+3d_2)}{n^3}-\frac{C_0}{20n^4}\\
			=&\frac{4(d_1+3d_2)}{n^2}\left(2-d_1-d_2-\frac{2}{n}\right)-\frac{C_0}{20n^4}\\
			\geq&\frac{2(c_0-2)}{n^3}-\frac{C_0}{20n^4}.
		\end{align*}
		Finally, for any $c_0>2$, there exists an integer $N(c_0)\geq n_1$ such that, for any $n\geq N(c_0)$, one has
		$$A_1A_2-pq>0.$$
		
		Finally,  for any $n\geq N(c_0)$, the positivity condition \eqref{beta1beta2>gamma1gamma2} holds due to our choice of $N(c_0)$.
	\end{proof}

	\begin{theorem}\label{mainthm for c_0}
		Suppose  that $p\geq1$ and   $q\geq 1$.  For each positive constant $c_0>2$, there exists a large  integer $N(c_0)$, such that
		for any  $n\geq N(c_0)$ and
		\begin{equation}\label{new region}
			\quad \frac{1}{p+1}+\frac{1}{q+1}\geq 1-\frac{2}{n}+\frac{c_0}{n^2},
		\end{equation}
		then system \eqref{lesys} has no positive solutions.
	\end{theorem}
	\begin{proof}
		Since $p\geq 1$ and $q\geq 1$, using \eqref{pqd1d2}, one has
		$$d_1\geq 0,\quad d_2\geq 0.$$
		The condition \eqref{new region} is equivalent to
		$$d_1+d_2<2-\frac{c_0}{n}.$$
		With out loss of generality, one  may assume that
		$d_1> d_2$ as we mentioned  before.

		Let $\varphi=\psi=\eta^m$ and choose   \eqref{rs k1 k2choice}  in Theorem \ref{estimate-theorem}, for $n\geq N(c_0)$,  the condition \eqref{beta1beta2>gamma1gamma2} holds by using Lemma \ref{lem:crucial estimate}.
		Meanwhile, it is not hard to check that
		\begin{equation}\label{rs-range}
			\frac{r+2}{q+1}=\frac{s+2}{p+1}
		\end{equation}
		from the choice that \eqref{rs k1 k2choice}. From \eqref{rs-range1}, for $n\geq N(c_0)$,  one has
		$$-1<r<0,\quad 0<s<p-1.$$

		Then, firstly, we 	let
		$0<\varepsilon<\min\{\alpha_1,\alpha_2\}$ small  enough to be chosen later  in Lemma \ref{estimates-error}.
		Using Theorem \ref{estimate-theorem} and Lemma \ref{estimates-error}, one has
		\begin{equation*}
			\left\{ \begin{aligned}
				&\beta_1\int_{\mathbb{R}^n} u^rv^{2p}\varphi\leq \gamma_1\int_{\mathbb{R}^n} u^{r+q+1}v^{p-1}\varphi+\varepsilon\int_{\mathbb{R}^n} u^{2q}v^s\varphi+C,\\
				&\beta_2\int_{\mathbb{R}^n} u^{2q}v^s\varphi\leq \gamma_2\int_{\mathbb{R}^n} u^{q-1}v^{p+s+1}\varphi+\varepsilon\int_{\mathbb{R}^n} u^{2q}v^s\varphi+C,
			\end{aligned} \right.
		\end{equation*}
		
		Then, one has
		\begin{align*}
			\beta_1\int_{\mathbb{R}^n} u^rv^{2p}\varphi\leq & \gamma_1\int_{\mathbb{R}^n} u^{r+q+1}v^{p-1}\varphi+\varepsilon\int_{\mathbb{R}^n} u^{2q}v^s\varphi+C\\
			\leq &\gamma_1\int_{\mathbb{R}^n} u^{r+q+1}v^{p-1}\varphi+\frac{\varepsilon\gamma_2}{\beta_2-\varepsilon}\int_{\mathbb{R}^n}u^{q-1}v^{p+s+1}\varphi+C.
		\end{align*}
		
		Using Lemma \ref{component-comparison} and \eqref{rs-range}, one has
		$$u^{q-1}v^{p+s+1}\leq C_1(n,d_1, d_2)u^{r+q+1}v^{p-1},$$
		where $C_1(n,d_1,d_2)$ is a positive constant depending on $n,d_1,d_2$.
		
		Then, we obtain that
		\begin{equation}\label{beta_2u^2qv^s}
			(\beta_2-\varepsilon)\int_{\mathbb{R}^n}u^{2q}v^s\varphi\leq \gamma_2\int_{\mathbb{R}^n}u^{q-1}v^{p+s+1}\varphi+C
		\end{equation}
		and
		\begin{equation}\label{beta_1u^rv^2p}
			\beta_1\int_{\mathbb{R}^n}u^rv^{2p}\varphi\leq \left(\gamma_1+\frac{\varepsilon\gamma_2C_1(n,d_1,d_2)}{\beta_2-\varepsilon}\right)\int_{\mathbb{R}^n}u^{r+q+1}v^{p-1}\varphi +C.
		\end{equation}
		
		With the help of H\"older's inequality, one has
		\begin{equation}\label{Holder's estimte}
			\int_{\mathbb{R}^n}u^{r+q+1}v^{p-1}\varphi\int_{\mathbb{R}^n}u^{q-1}v^{p+s+1}\leq \int_{\mathbb{R}^n}u^rv^{2p}\varphi\int_{\mathbb{R}^n}u^{2q}v^s\varphi.
		\end{equation}
		Since $\beta_1\beta_2>\gamma_1\gamma_2$, we can choose $\varepsilon$ small enough such that
		$$\beta_1\left(\beta_2-\varepsilon\right)>\left(\gamma_1+\frac{\varepsilon\gamma_2C_1(n,d_1,d_2)}{\beta_2-\varepsilon}\right)\gamma_2.$$
		
		With the help of \eqref{Holder's estimte}, multiplying \eqref{beta_2u^2qv^s} with \eqref{beta_1u^rv^2p}, we can obtain that
		\begin{equation}\label{u^sv^2pu^2qv^s}
			\int_{\mathbb{R}^n}u^{2q}v^s\varphi\int_{\mathbb{R}^n}u^rv^{2p}\varphi\leq C\int_{\mathbb{R}^n}u^{r+q+1}v^{p-1}\varphi+C\int_{\mathbb{R}^n}u^{q-1}v^{p+s+1}\varphi+C.
		\end{equation}
		
		Using Lemma \ref{component-comparison} and \eqref{rs-range} again, one has
		$$u^{r+q+1}v^{p-1}\leq  Cu^{2q}v^s,\quad u^{q-1}v^{p+s+1}\leq Cu^{2q}v^s, \quad u^{r}v^{2p}\leq Cu^{2q}v^s.$$
		Applying the above estimate into \eqref{u^sv^2pu^2qv^s}, one has
		$$\int_{\mathbb{R}^n}u^{2q}v^s\varphi\int_{\mathbb{R}^n}u^rv^{2p}\varphi\leq C\int_{\mathbb{R}^n}u^{2q}v^s\varphi+ C$$
		which yields that
		$$\int_{\mathbb{R}^n}u^rv^{2p}\varphi\leq C+ \frac{C}{\int_{\mathbb{R}^n}u^{2q}v^s\varphi}\leq C+\frac{C}{\int_{\mathbb{R}^n}u^rv^{2p}\varphi}.$$
		Thus, one obtain that
		\begin{equation*}
			\int_{\mathbb{R}^n} u^{r}v^{2p}\varphi\leq C
		\end{equation*}
		where $C$ is a positive constant which is independent of $u$ and $v$. Hence replacing the solution $(u(x),v(x))$ by $(R^\alpha u(Rx),R^\beta v(Rx))$ and changing variables, we have
		\begin{equation}\label{universalestimats2}
			\int_{B_R} u^{r}v^{2p}\leq CR^{\frac{-(d_1+3d_2)-(2-d_1-d_2)n}{d_1+d_2}}.
		\end{equation}
		Letting $R\rightarrow\infty$ in \eqref{universalestimats2}, we obtain a contradiction due to \eqref{d1d2range}.
	\end{proof}

	\section{Numerical estimate and proof of main theorem}\label{sec:numercial}
	In this section,we are going to do some numerical estimate with the help of Mathematica.
	\begin{lemma}\label{lem:c_0=4}
		Under the same assumption in Lemma \ref{lem:crucial estimate} and choose $c_0=4$. Then, for $n\geq 13$,
		the  positivity  condition
		\eqref{beta1beta2>gamma1gamma2} holds.
	\end{lemma}
	\begin{proof}
		From the argument in Lemma \ref{lem:crucial estimate}, we only need to check that, for $n\geq 13$,
		$$A_1A_2-pq>0.$$
		To prove this estimate, we have the following  decomposition by using Mathematica:
		\begin{align*}
			&A_1A_2-pq\\
			=&\frac{4(2-d_1-d_2)(d_1+3d_2)}{(n-1)^2}+\frac{T(d_1,d_2)}{(n-1)^3}\\
			&+\frac{\sum^{14}_{i=0}R_i(d_1,d_2)n^{i}}{2(n-2d_1)^2(n-2d_2)^2(n-1)^4(n+2d_1)(n+2d_2)A_3A_4}
		\end{align*}
		where $A_3$ and  $A_4$ come from \eqref{A3def} and \eqref{A4def},
		\begin{align*}
			T(d_1,d_2)=&\frac{d_1^2}{2}(19d_1-76-125d_2)\\
			&+\frac{d_1}{2}(60+296d_2-127d_2^2)+\frac{d_2}{2}(-76+36d_2-23d_2^2),
		\end{align*}
		and $R_{i}(d_1,d_2)$ are some polynomials 	in $d_1,d_2$ with integer coefficients.
		
		With the help of Mathematica, we are able to establish lower bounds for the coefficients of the powers of $n$, under the constraint $0\leq d_2<d_1<2-d_2$. Then, for $n\geq 2$,
		one  has
		\begin{equation}\label{A_3estimate}
			A_3\geq 4n^4-16n^3-60n^2-n
		\end{equation}
		and
		\begin{equation}\label{A_4estimate}
			A_4\geq 4n^4-16n^3-16n^2-50n-12.
		\end{equation}
		
		Under the constraint $0\leq d_2<d_1<2-d_2$, one has
		$$T(d_1,d_2)\geq -8(d_1+3d_2)$$
		by using Mathematica.
		
		The precise expressions for $R_i(d_1,d_2)$ can be obtained using Mathematica, but they are too lengthy to present here. For our purposes, it suffices to provide the lower bounds computed with Mathematica, as shown below.
		Under the constraint
		\begin{equation}\label{constriantfor R_14-10}
			0\leq d_2<d_1<2-d_2,
		\end{equation}
		one has
		\begin{align*}
			\frac{R_{14}(d_1,d_2)}{32}\geq&-31(4d_1+12d_2)\\
			\frac{R_{13}(d_1,d_2)}{32}\geq&-40(4d_1+12d_2)\\
			\frac{R_{12}(d_1,d_2)}{32}\geq&-803(4d_1+12d_2)\\
			\frac{R_{11}(d_1,d_2)}{32}\geq&-1982(4d_1+12d_2)\\
			\frac{R_{10}(d_1,d_2)}{32}\geq&-200(4d_1+12d_2).
		\end{align*}
		As for the remaining terms, a precise computation under constraint \eqref{constriantfor R_14-10} is prohibitively time-consuming. Moreover, these terms are of lower order for sufficiently large $n$. We therefore confine ourselves to a rough lower bound valid over a larger domain:
		\begin{equation*}
			0 \leq d_2 < 1, \quad 1/2 < d_1 < 2.
		\end{equation*}
		The results are as follows:
		\begin{align*}
			\frac{R_{9}(d_1,d_2)}{32}\geq&-6\times 10^4(4d_1+12d_2)\\
			\frac{R_{8}(d_1,d_2)}{32}\geq&-1.5\times 10^5(4d_1+12d_2)\\
			\frac{R_{7}(d_1,d_2)}{32}\geq&-1600(4d_1+12d_2)\\
			\frac{R_{6}(d_1,d_2)}{32}\geq&-8\times 10^5(4d_1+12d_2)\\
			\frac{R_{5}(d_1,d_2)}{32}\geq&-9\times 10^5(4d_1+12d_2)\\
			\frac{R_{4}(d_1,d_2)}{32}\geq&-8\times 10^4(4d_1+12d_2)\\
			\frac{R_{3}(d_1,d_2)}{32}\geq&-3\times 10^6(4d_1+12d_2)\\
			\frac{R_{2}(d_1,d_2)}{32}\geq&-1.3\times 10^5(4d_1+12d_2)\\
			\frac{R_{1}(d_1,d_2)}{32}\geq&-40(4d_1+12d_2)\\
			\frac{R_{0}(d_1,d_2)}{32}\geq&-3(4d_1+12d_2)
		\end{align*}
		For $n\geq 13$, combining with \eqref{A_3estimate} with \eqref{A_4estimate},  one has
		\begin{align*}
			&A_1A_2-pq\\
			\geq&\frac{4(d_1+3d_2)}{(n-1)^4}\left(\frac{4(n-1)^2}{n}-2(n-1)\right)\\
			&-\frac{4(d_1+3d_2)}{(n-1)^4}\cdot\frac{32(31n^{14}+40n^{13}+803n^{12}+1982n^{11}+200n^{10}+4\times 10^7n^9)}{2(n-4)^2(n-2)^2n^2(4n^4-8n^3-37n^2-n)(4n^4-8n^3-16n^2-13n-6)}.
		\end{align*}
		Using Mathematica,  for $n\geq 35$,  one has
		$$\frac{4(n-1)^2}{n}-2(n-1)-\frac{16(31n^{14}+40n^{13}+803n^{12}+1982n^{11}+200n^{10}+4\times 10^7n^9)}{(n-4)^2(n-2)^2n^2(4n^4-16n^3-60n^2-n)( 4n^4-16n^3-16n^2-50n-12)}>0.$$
		Thus,   for $c_0=4$ and $n\geq 35$, one has
		$$A_1A_2-pq>0.$$
		
		For each integer  with $13\leq n\leq 35$, we use Mathematica  to  directly verify  that
		$$A_1A_2-pq>0$$
		under the constraints $$0\leq d_2<1,\quad \frac{1}{2}<d_1<2-d_2-\frac{4}{n}.$$
		
		Combining with Lemma \ref{lem:crucial estimate}, the positivity condition \eqref{beta1beta2>gamma1gamma2} holds.
	\end{proof}
	
	As for the case $5\leq n\leq 12$, we need modify the choice of $k_1,k_2$ in \eqref{rs k1 k2choice}.
	Firstly,  for $5\leq n\leq 12$,under the constraints
	$$0\leq d_2<d_1<2-d_2-\frac{4}{n}$$
	it is hard to check that
	$$-\frac{n}{n-2}<-\frac{3(d_1-d_2)}{n-2d_2}<0,\quad 0<\frac{d_1-d_2}{n-2d_1}<1.$$
	Then, we choose
	\begin{equation}\label{rs k1 k2choice2}
		\left\{
		\begin{aligned}
			r=&-\frac{3(d_1-d_2)}{n-2d_2}, \quad s=\frac{d_1-d_2}{n-2d_1}\\
			k_1=&\frac{n}{2(n-1)}\left(r-1-\sqrt{\frac{(1-r)((n-2)r+n)}{n}}\right)+\epsilon_0\\
			k_2=&\frac{n}{2(n-1)}\left(s-1-\sqrt{\frac{(1-s)((n-2)s+n)}{n}}\right)+\epsilon_0\\
			\theta_1=&\frac{\rho_1-r}{p},\quad 	\theta_2=\frac{\rho_2-s}{q}
		\end{aligned}\right.
	\end{equation}
	where  $\epsilon_0$ is positive and small enough, and
	$$\rho_1=\frac{n}{n-1}\left(\frac{3r}{2}-\frac{(n+2)k_1}{n}\right),\quad \rho_2=\frac{n}{n-1}\left(\frac{3s}{2}-\frac{(n+2)k_2}{n}\right).$$
	
	\begin{lemma}\label{lem:5leq nleq 12}
		Choose \eqref{rs k1 k2choice2} in Theorem \ref{estimate-theorem}. For each integer $n$ with $5\leq n\leq 12$, the positivity condition \eqref{beta1beta2>gamma1gamma2} holds.
	\end{lemma}
	\begin{proof}
		Firstly, we set \(\epsilon_0 = 0\). In this case, although \(\alpha_1 = \alpha_2 = 0\), we shall show that
		\begin{equation}\label{positive condition1}
			\gamma_1 > 0, \quad \gamma_2 > 0, \quad \beta_1 > 0, \quad \beta_2 > 0, \quad \beta_1\beta_2 - \gamma_1\gamma_2 > 0.
		\end{equation}
		Since these parameters depend continuously on \(\epsilon_0\), we may choose a sufficiently small \(\epsilon_0 > 0\) such that
		\[
		\alpha_1 > 0 \quad \text{and} \quad \alpha_2 > 0,
		\]
		while still maintaining the inequalities in \eqref{positive condition1}.
		
		Hence, we only need to focus on the case \(\epsilon_0 = 0\).
		By our choice in \eqref{rs k1 k2choice2}, it is straightforward to verify that
		\[
		\rho_1 > 0, \quad \rho_2 > 0.
		\]
		Substituting the expressions for \(\theta_1\) and \(\theta_2\) from \eqref{rs k1 k2choice2} into \eqref{delta1} and \eqref{delta2}, we obtain
		\[
		\delta_1 = \frac{r^2}{p} + r + \frac{p-1}{p} \rho_1^2, \quad
		\delta_2 = \frac{s^2}{q} + s + \frac{q-1}{q} \rho_2^2.
		\]
		Clearly, \(\delta_2 > 0\). Moreover, for \(5 \leq n \leq 12\), we use Mathematica to verify that
		\[
		\delta_1 > 0
		\]
		under the constraint \(0 \leq d_2 < d_1 < 2 - d_2 - \frac{4}{n}\).
		Therefore, we conclude that
		\[
		\gamma_1 > 0, \quad \gamma_2 > 0.
		\]
		Again with the aid of Mathematica, we confirm that
		\[
		A_1 > 0, \quad A_2 > 0, \quad A_1 A_2 - p q > 0
		\]
		under the same constraint \(0 \leq d_2 < d_1 < 2 - d_2 - \frac{4}{n}\).
		This completes the proof of the lemma.
	\end{proof}

	\quad
	
	\quad
	
	{\bf Proof of Theorem \ref{mainthm}:}
	
	Following the same  argument of Theorem \ref{mainthm for c_0}, we finish the proof by using Lemma \ref{lem:c_0=4} and Lemma \ref{lem:5leq nleq 12}.

	\section*{Acknowledgements}
	This work was partially carried out during the first author's visit to The Chinese University of Hong Kong and the second author's visit to Zhengzhou University, whose hospitality they wish to acknowledge.  J.C. Wei is supported by National Key
	R$\&$D Program of China 2022YFA1005602, and Hong Kong General Research Funds “New frontiers in singularity formations of nonlinear partial differential equations" and “On Fujita equation in critical and supercritical regime".

\end{document}